\documentclass[english, a4paper, 12pt]{article}

\usepackage{enumerate}
\usepackage[utf8]{inputenc}
\usepackage[T1]{fontenc}								
\usepackage[top=25mm, left=35mm, right=25mm, bottom=25mm, includefoot]{geometry}		
\usepackage{enumerate} 
\usepackage{graphicx}
\usepackage{amsmath, amsfonts, amssymb, amsthm} 
\usepackage{babel}
\usepackage{times}					%pismo times %
\usepackage{verbatim}				%prostredie na komenty
\usepackage{float} %pozicia obrazkov

\newtheorem{thm}{Theorem}
\newtheorem{cor}[thm]{Corollary}
\newtheorem{lem}{Lemma}

\newtheorem*{Q}{Question}

\theoremstyle{definition}

\newtheorem*{definition*}{Definition}
\newtheorem{note}{Note}
\newtheorem*{note*}{Note}

\theoremstyle{remark}
\begin{document}
%\hfill \today \\

\title{Distributional Chaos and Dendrites}

\author{Zuzana Roth}

%\keywords{Chaos, Dendrites, Distributional Chaos, Horseshoes}

%\subjclass[2010]{}

\maketitle

\begin{abstract}

Many definitions of chaos have appeared in the last decades and with them the question if they are equivalent in some more specific spaces.
Our focus will be distributional chaos, first defined in 1994 and later subdivided into three major types (and even more subtypes). 
These versions of chaos are equivalent on a closed interval, but distinct in more complicated spaces. Since dendrites have much in common with the interval, we explore whether or not we can distinguish these kinds of chaos already on dendrites.
In the end of the paper we will also briefly look at the correlation with other types of chaos.

\end{abstract}


\section{Introduction}
When we look in the literature for some examples   showing that (in general) various types of Distributional Chaos (DC) are not equivalent (\cite{Stef},\cite{RRH}), we find systems with more complicated structure, spaces with subsets homeomorphic to the circle or at least some triangular (skew product) maps on unconnected spaces. On the other hand, we know that all types of DC are equivalent on the interval, trees and graphs (\cite{MM}, \cite{KKM}, \cite{HaM}, \cite{RL}).
Since dendrites, trees and intervals are so much alike, can it be that all types of  DC are  equivalent on dendrites? We will answer  this question in  section \ref{sec:DC3}. \\
Also if we look into the  literature we can find  different relationships between the original definition of distributional chaos (which required the existence of  DC-pairs) and other types of chaos. How will the situation on dendrites change, if we will require an uncountable DC-scrambled set? In the end of  section \ref{sec:DC3} you will find the relation between  DC-pairs and  uncountable DC-scrambled sets, and in section  \ref{sec:nesp.mn} the relation with other types of chaos.

%===DEFINICE===
%\widehat{y}x
\section{Terminology}\label{sec:term}
We will use the following notation through the whole paper,  if  not indicated otherwise.
Let $(X,d)$ be a non-empty compact metric space.  A pair $(X,f)$, where $f$ is a continuous self-map acting on $X$, is called a \emph{(topological) dynamical system}.  The \emph{orbit} of a point $x\in X$ is the set $\{f^n(x):n\geq0\}$.
%Let $(X,f)$  be a dynamical system on a compact metric space. 

% Li-Yorke
%\begin{definition}[Li-Yorke chaos]\label{LYC}
A pair of two different points $(x, y)\in X^2$ is \emph{scrambled} or \emph{Li-Yorke} if 
\begin{equation}\label{Li1}\liminf_{k\to\infty}d(f^k(x),f^{k}(y))=0\end{equation} and 
\begin{equation}\label{Li2}\limsup_{k\to\infty} d(f^k(x),f^{k}(y))>0.\end{equation}
A subset $S \subset X$ is \emph{LY-scrambled} if it contains at least 2 distinct points  and every pair of distinct points in $S$ is scrambled. According to the size of $S$ we say that $f$ is LY$_2$, if $S$ contains a scrambled pair, LY$_\infty$,  if $S$ is an infinite LY-scrambled set, or LY$_u$, if $S$ is an uncountable LY-scrambled set. 
The system $(X,f)$ is usually called \emph{Li-Yorke chaotic} if there exists an uncountable LY-scrambled set.
%We call a pair $(x, y)\in X^2$ \emph{proximal} if (\ref{Li1}) holds (otherwise we say $(x, y)\in X^2$ is \emph{distal}). If (\ref{Li2}) does not hold, i.e. $$\limsup_{k\to\infty} d(f^k(x),f^{k}(y))=0,$$ we say that $(x, y)\in X^2$ is \emph{asymptotic}. A pair of points is scrambled simply if it is proximal but not asymptotic. A dynamical system $(X,f)$ is \emph{distal} if every pair of distinct points in $X$ is distal.
%\end{definition}
\\

%DC - chaos
%\begin{definition}[Distributional chaos]\label{DC}

As we already mentioned, at the beginning, there was a definition of one kind of  DC (see \cite{SS}), this type is called DC1 in these days, later (\cite{BSS})  that type was divided into 3 different types DC1, DC2 and DC3, different in general, but the same in the interval. It also turned out, that DC3 can be a really weak and unstable type of chaos, so in \cite{RRH} appeared a new kind of DC, namely DC2$\frac12$ which as was shown, fixed those problems, but in general it is essentially weaker than DC2.
 (There is also DC1$\frac12$ see \cite{Down}, but we will not discuss this kind in this paper.)  

For a pair $(x, y)$ of
points in $X$, define the \emph{lower distribution function} generated by $f$ as
\begin{equation}\label{LoDF}
\Phi_{(f,x, y)}(\delta)=\displaystyle\liminf_{n\to\infty}\frac{1}{n}\#\{0 \le k \le n;d(f^k(x),f^{k}(y))<\delta\},
\end{equation}
and the \emph{upper distribution function} as 
\begin{equation}\label{UpDF}
\Phi^*_{(f,x, y)}(\delta)=\displaystyle\limsup_{n\to\infty}\frac{1}{n}\#\{0 \le k \le n;d(f^k(x),f^{k}(y))<\delta\},
\end{equation}
where $\#A$ denotes the cardinality of the set $A$.\\ 
A pair $(x, y)\in X^2$ is called \emph{distributionally scrambled of type 1} (or a DC1 pair) if
$$\Phi^*_{(f,x, y)}(\delta)=1, \mbox{  for every $0<\delta\le \text{diam }X$}$$
  and  
  $$ \Phi_{(f,x, y)}(\epsilon)=0, \mbox{  for some  }0<\epsilon\le \text{diam }X ,$$
\emph{distributionally scrambled of type 2} (or a DC$2$ pair) if 
$$\Phi^*_{(f,x, y)}(\delta)=1, \mbox{  for every $0<\delta\le \text{diam }X$}$$
  and  
  $$ \Phi_{(f,x, y)}(\epsilon)< 1,\mbox{  for some  }0<\epsilon\le \text{diam }X ,$$
\emph{distributionally scrambled of type 3} (or a DC$3$ pair) if $$ \Phi_{(f,x, y)}(\delta)<\Phi^*_{(f,x, y)}(\delta), \mbox{  for every $\delta$ in some interval $(a,b),$ where } 0\leq a<b\le \text{diam }X.$$
A subset $S$ of $X$ is \emph{distributionally scrambled of type $i$} (or a DC$i$ set), where $i=1,2,3$, if every pair of distinct points in $S$ is a DC$i$ pair.  
Originally, the dynamical system $(X,f)$ was called \emph{distributionally chaotic of type $i$} (a DC$i$ system), where $i=1,2,3$, if there was a DC$i$ pair (DC$i_2$), %we will denote this type of chaos as DC$i_2$.
later the focus was moved to  an uncountable $i$  - distributionally scrambled sets (DC$i_u$).
%Currently the dynamical system is considered as  \emph{distributionally chaotic of type $i$}, if there is an uncountable distributionally scrambled set $S\subset X$ of type $i$, we will denote this type as DC$i_u$.
%\end{definition}

We can also define both distribution functions at $0$ as limits:
 $\Phi_{(f, x, y)}(0)=  \displaystyle\lim_{\delta\to0^+}\Phi_{(f,x, y)}(\delta)$ and 
 $\Phi^*_{(f, x, y)}(0)=\displaystyle\lim_{\delta\to 0^+}\Phi^*_{(f,x, y)}(\delta)$. Then $(x, y)\in X^2$ is called \emph{distributionally scrambled of type $2\frac12$} (or a DC$2\frac12$ pair) if
$$\Phi_{(f,x, y)}(0)<\Phi^*_{(f,x, y)}(0).$$
We define  DC$2\frac12$ sets and DC$2\frac12$ systems in the same way as for the other 3 versions of distributional chaos.

\begin{note}\label{DCnote}
Notice that $\Phi_{(f,x, y)}(\delta)= 1 - \displaystyle\limsup_{n\to\infty}\frac{1}{n}\#\{0 \le k \le n;d(f^k(x),f^{k}(y)) \ge \delta\} .$
\end{note}

%horseshoe
%\begin{definition}[Horseshoe]\label{HS}
Suppose that there are disjoint compact sets $A,B \subset X$ such that
\begin{equation}\label{podkova}
f(A) \cap f(B) \supset A \cup B.
\end{equation}
Then we say that $f$ has a  {\em horseshoe} or that $A$ and $B$ form a horseshoe for $f$. (Since several definitions of horseshoes have appeared over time,  this general case is  in other literature also known as a {\em strict general horseshoe}).
If $X$ is a graph or dendrite and there are arcs $A,B \subset X$ which intersect at most in their  endpoints, statisfying (\ref{podkova}), then we say that $A, B$ form an {\it arc horseshoe} for $f$. Moreover if the sets $A, B$ are disjoint, %and satisfy (\ref{podkova}), then 
we say that $A, B$ form a {\it strict arc horseshoe} for $f$.
%\end{definition}
(For other types of horseshoes see \cite{Koc} or \cite{KKM}.)

%omega chaos
%\begin{definition}[omega chaos]\label{omega}
The set of limit points of the sequence
$(f^n(x))_{n\in \mathbb N}$ is called the  $\omega${\it-limit set of the point x under f} and denoted by $\omega_f (x)$.
A set $S \subset X$ is called $\omega$-scrambled for $f$ if it contains at least two points and  for any 2 distinct points $x, y \in S$ the following conditions hold:
\begin{enumerate}[1)]
\item $\omega_f (x)\setminus \omega_f(y)$ is uncountable,
\item  $\omega_f (x) \cap \omega_f(y)$ is nonempty,
\item  $\omega_f (x)$ is not contained in the set of periodic points.
\end{enumerate}
%\end{definition}
%dendrite
%\begin{definition}\label{dendrite}
A {\it dendrite} is a locally connected continuum (compact connected metric space) %A space D is connected if any two points in D can be connected by a curve lying wholly within D
 containing no subset homeomorphic to the circle (no simple closed curve).
A point of a dendrite is called a {\it branch point} if it is the endpoint of three arcs with disjoint interiors.
A point of a dendrite is called an {\it end point} if for every one of its neighborhoods $U$  there exists a neighborhood $V \subset U$ with a one-point boundary.
% A point of a dendrite is called a {\it end point} if it is not the common end point of two arcs with disjoint interiors.
%\end{definition}

\bigskip
%-
%
% Sekcie s vetami a dokazmi :D

%\section{DC3 without DC2}\label{sec:DC3}
\section{ Equivalence of different types of DC on dendrites}\label{sec:DC3}
%When we look in the literature for some examples   showing that in general different types of DC are not equivalent (\cite{Stef},\cite{RRH}), we find systems with more complicated structure, spaces with subsets homeomorphic to the circle or at least some triangular maps on unconnected spaces.On the other hand, we know that all types of DC are equivalent on the interval, trees or graphs (\cite{MM}, \cite{KKM}, \cite{HaM}, \cite{RL}).
%Since dendrites, trees and intervals are so much alike, can it be that all types of  DC are  equivalent on dendrites? We will answer  this in the next two subsections.

\subsection{ DC1 and DC2}\label{subsec:DC1}

(Un)fortunately, 
there are already several articles showing that we can associate the Gehman dendrite (the topologically unique dendrite whose set of end points is homeomorphic to the Cantor set) and a map $g$ on it with a shift space and the shift map $ \left( \{0,1\}^{\mathbb N_0}, \sigma \right)$. Moreover, we can also build a  subdendrite of the Gehman dendrite associated with any  subshift of the full 2-shift (see e.g. \cite{Koc}, \cite{KKM}, \cite{TD}). That means we can use not just general results but also  results  known from shift spaces.

\begin{figure}[h]
 \centering
 \includegraphics[height=4.2cm]{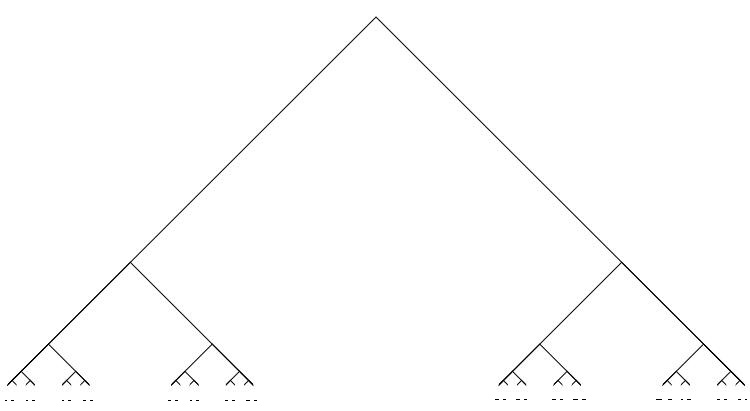}  
  \caption{Gehman dendrite}
   \end{figure}

\begin{lem}[%Drwi\k{e}ga
 \cite{TD} Lem 4.1]
If $X \subset  \{0,1\}^{\mathbb N_0}$
is a subshift then there is a subdendrite $\mathcal G_X$ of the Gehman dendrite $\mathcal G$ invariant under $g$. Let $ \mathcal E_X$ be the set of end points of $\mathcal G_X$, then $( \mathcal E_X, g |_{\mathcal E_X} )$ is topologically conjugate to $(X, \sigma).$
%the set $\mathcal G_X$ is an $g|_{\mathcal G_X}$-invariant subdendrite of the Gehman dendrite $\mathcal G$.
\end{lem}

\begin{note}
Notice, that  all $x \in \mathcal G_X \setminus \mathcal E_X$ are eventually fixed and so all the interesting dynamics happen exclusively in $\mathcal E_X$ and so if we are looking for DC-pairs/sets we need only look in $\mathcal E_X$ (for more details see the proof of Lem. 4 in  \cite{Koc}).
\end{note}

\begin{thm}[%Downarowicz 
\cite{Down} Th 1.1]\label{ThDown}
 Assume that a topological dynamical system $(X, T)$ has positive
topological entropy. Then the system possesses an uncountable DC2-scrambled set.
\end{thm}

\begin{thm}[%Piku\l a 
\cite{Pikula} Ex 4.1]\label{ThPik}
There exists a subshift with positive topological entropy without DC1 pairs.
\end{thm}

As a combination of the previous theorems and lemma we obtain the following result:

\begin{thm}\label{ThDC1}
DC2 does not imply DC1 on dendrites.
\end{thm}

%%%%%%%%%%%%
%      DC3 without DC2
%%%%%%%%%%%%%
\subsection{ DC2 and DC3}\label{subsec:DC3}
While for DC1 and DC2 we could go directly to known results on subshifts of the full 2-shift, for DC3 and DC2 we will start with a less complicated structure. The dendrite $\mathcal D$ will be a comb-style dendrite and we will use a level-inductive aproach for the construction of our dendrite.\\
In the following construction  $ I~$denotes the unit interval $  \langle0,1\rangle \times \{0\},$ $(s,t)=\{s\} \times \{t\}$ and if we will need a metric, $d$ will be the max metric in $\mathbb R^2.$\\
\textbf{Construction of the dendrite $\mathcal D$:}\\
% denote $$
$$\mathcal D = I \cup \left ( \bigcup_{n \in \mathbb N}  \left (  Z^{(n)}   \times \langle 0, h_n \rangle \right ) \right ) $$ 
%where
where $ Z^{(n)} = \left \{ z_j ^{(n)}  = \dfrac{j}{3^n} : j \in J_n \right\} ,\ 
%J_n = \left ( \{ j \}_{j=1}^{3^n} \setminus \{ 3\cdot i \}_{i=1}^{3^{n-1}}  \right) \ , 
J_n = \left \{ j : \ 1\le j  \le {3^n} \wedge \ 3 \nmid j  \right \} ,
\  h_n = \dfrac{1}{3^{n}} % \ $ and  $ I~$denotes a unit interval $  \langle0,1\rangle \times \{0\} 
.$ \\
The set $ \left (\left ( \bigcup_{n \in \mathbb N}  \left (  Z^{(n)}   \times \{ h_n \} \right ) \right ) \cup \left\{(0,0), (1,0)\right\}  \right ) $ is the set of endpoints of $\mathcal D.$\\
 Moreover let's denote $l_n=\# \left ( \bigcup_{i=1}^n    Z^{(i)}   \right ) = 3^n -1$. \\

%\begin{wrapfigure}{c}{0.65\textwidth}
%%\begin{picture}(0,0)
%\includegraphics[width=0.65\textwidth]{./img/DC3Dendrite.jpeg}
%%\end{picture}
%\caption{A dendrite $\mathcal D$ - first 4 layers.}
%\end{wrapfigure}

We can easily imagine the dendrite construction by an inductive approach:
 we start with 1 horizontal line (the interval $ I$). Then we add levels of  ``spikes:''  in every next level we add 2 equally distributed spikes between every 2 ``old'' spikes.
 %``border/end''- points. 
 Spikes in the $n$-th level emanate from the points of  $Z^{(n)} $ and have  height $h_n= {1}/{3^{n}} $. So the first level has 2 ``spikes''  at $\frac13 $  and $\frac23$ with height $\frac13$, in the second level we add 6 spikes of  height $\frac19$ at $\frac19, \frac29, \frac49, \frac59, \frac79, \frac89$ (all $i/9$ between 0 and 1 except for $i$-s which are multiples of 3, those are already in the previous level, so after finishing the second level we have 8 spikes total) and so on. \\

\begin{figure}[h]
%\begin{picture}(0,0)
 \centering
\includegraphics[width=0.75\textwidth]{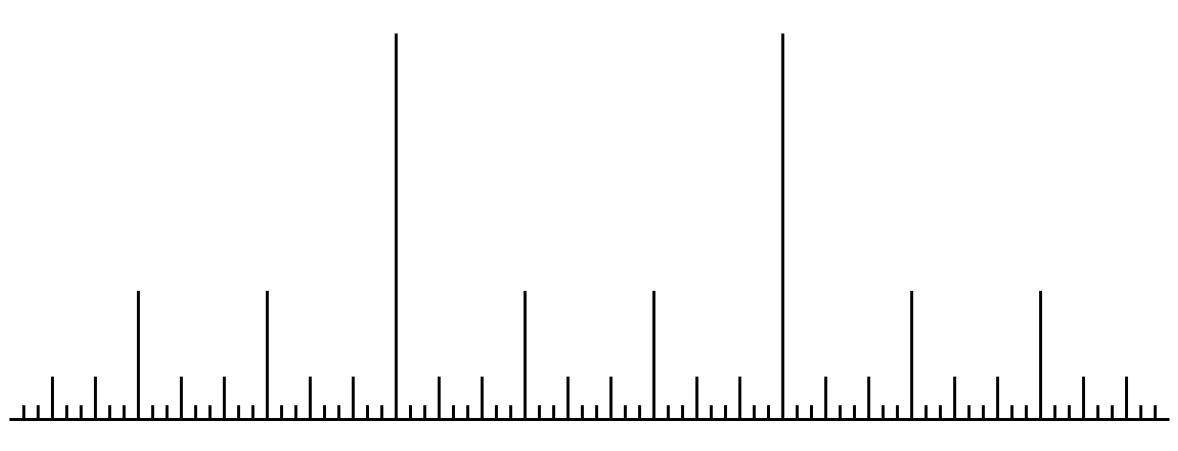}
%\end{picture}
\caption{Construction of the dendrite $\mathcal D$ - first 4 levels.}
\label{dendriteDC3}
\end{figure}

%zobrazenie
\textbf{Construction of the map $f: \mathcal D \to \mathcal D$:}\\
Let's first construct helping maps. For  $n \in \mathbb N $ and $ j \in J_n $ denote by $\phi_{(n,j)} , \,  \psi_{(n,j)}$ the increasing linear functions such that:

%PSI 
%spodna cast 
\begin{equation*}
		\psi_{(n,j)}\left( \left \langle 0, \frac23h_n \right\rangle \right)
			= \left \{ \begin{array}{l  l}
						  \mbox{for odd  $n$} &
						 \left \{
						   \begin{array}{l  l}
							\left\langle z_j^{(n)}, z_{j+1}^{(n)}\right \rangle , & \mbox{  if  $j=3a+1$ } \\[0.75em]
							\left\langle z_j^{(n)}, z_{j+2}^{(n)}\right \rangle , & \mbox{ if }  j=3a+2 \neq l_n  \\[0.75em]
							\left\langle z_{l_n}^{(n)}, z_{l_{n+1}}^{(n+1)}\right \rangle,   & \mbox{  if } j  =l_n = 3^n -1  \\ [0.75em]
							\end{array} \right. \\	
							& \\[0.75em]
				
						 \mbox{for even $n$}  &
						 \left \{
						   \begin{array}{l  l}
							\left\langle z_{j-1}^{(n)}, z_{j}^{(n)}\right \rangle  , &  \mbox{  if  $j=3a+2$ } \\[0.75em]
							\left\langle z_{j-2}^{(n)}, z_{j}^{(n)}\right \rangle , & \mbox{ if }  j=3a+1 \neq 1   \\[0.75em]
							\left\langle z_{1}^{(n+1)}, z_{{1}}^{(n)}\right \rangle,   & \mbox{  if } j = 1    \\ 
							\end{array} \right. 
  						\end{array}
								\right. 
								 a \in \mathbb N_0
	\end{equation*}
	
	%PHI
	\begin{equation*}
		\phi_{(n,j)}\left( \left \langle  \frac23h_n, h_n \right\rangle \right)
			= \left \{ \begin{array}{l  l}
						  \mbox{for odd  $n$} &
						 \left \{
						   \begin{array}{l  l}
							 \langle  0 , h_n \rangle  , & \mbox{  if } j \neq 3^n -1 \\ [0.5em]
							 \langle  0 , h_{n+1} \rangle,   & \mbox{  if } j = 3^n -1    \\ \end{array} \right. \\
							\\
						\mbox{for even $n$}  &
						 \left \{
						   \begin{array}{l  l}
							 \langle  0 , h_n \rangle , & \mbox{  if } j \neq 1 \\[0.5em]
							 \langle  0 , h_{n+1} \rangle,   & \mbox{  if } j = 1     \end{array} \right. \\
  						\end{array}
								\right.
	\end{equation*} \\
	%f
	Now we use the functions $\psi_{(n,j)}$ and $\phi_{(n,j)}$ to construct $f$ in three parts:
	\begin{enumerate}
		\item $f|_I=id$. \\
		\item for $(x,y) \in \left ( \{z_j^{(n)}\}  \times ( 0, \frac23 h_n \rangle \right)$: \\
		$f(x,y)=f\left( z_j^{(n)},y\right)=\left( \psi_{(n,j)}(y) ,0  \right)$
		\item for $(x,y) \in \left (\{ z_j^{(n)}\}  \times ( \frac23 h_n, h_n \rangle \right)$: \\
		$f(x,y)=f\left( z_j^{(n)},y\right)=\left( \psi_{(n,j)}(\frac23h_n) ,\phi_{(n,j)}(y)  \right)$

	\end{enumerate}
%obrazok
%
\begin{figure}[H]
 \centering
\includegraphics[width=0.70\textwidth]{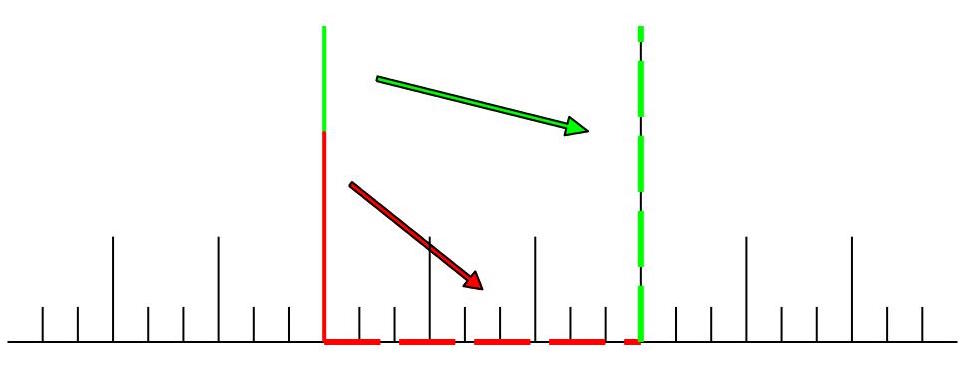}
\includegraphics[width=0.70\textwidth]{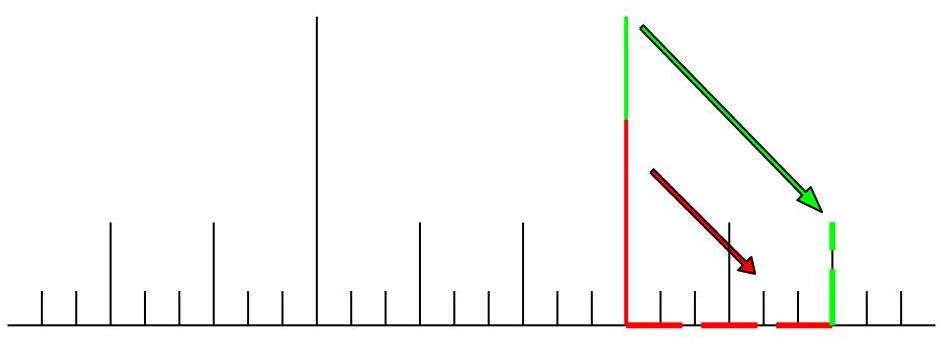}
\caption{Mapping ``spikes'' in odd levels by the  map $f$ .}
\label{dendriteDC3a}
\end{figure}
\begin{figure}[H]
 \centering
\includegraphics[width=0.85\textwidth]{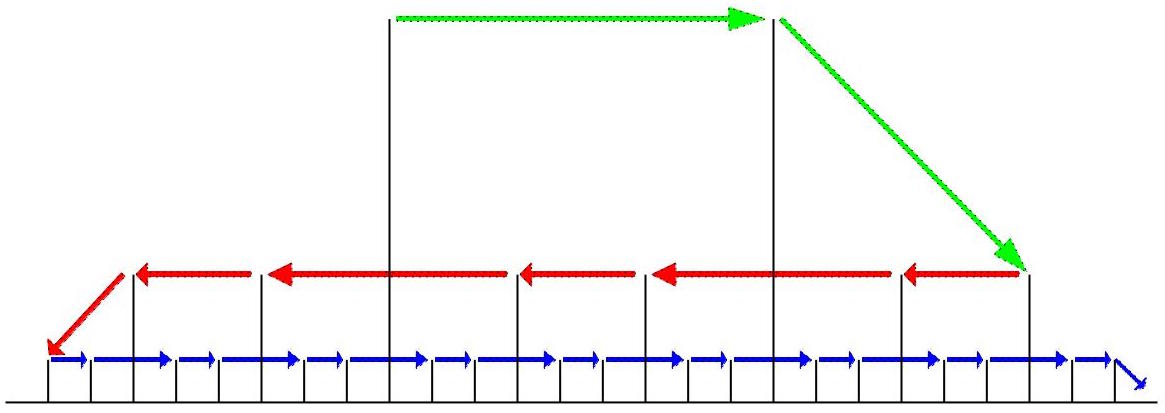}
\caption{ The map $f$- direction of mapping spikes/ trajectory of the point $\left(\frac13, \frac13 \right)$.}
\label{dendriteDC3b}
\end{figure}

\begin{lem}\label{spoj+}
	\begin{enumerate}[a)]
	%%aa
	\item\label{aa} The set $ \mathcal D$ is a dendrite.
	%a
	\item\label{a} The map $f$ is continuous. 
	%b
	\item\label{be} If $x  \in \mathcal D \setminus  \left ( \bigcup_{n \in \mathbb N}  \left ( Z^{(n)}   \times \{ h_n \} \right ) \right ) $ then $x$ is an eventually  fixed point. 
	%c
	\item \label{ce} If  $x, y \in   \bigcup_{n \in \mathbb N}  \left (  Z^{(n)}   \times \{ h_n \}  \right ) $, then 
$\displaystyle\lim_{n \to \infty} d (f^n(x), f^n(y)) = 0.$
	\end{enumerate}
\end{lem}

\begin{proof}
Proofs of  \ref{aa}), \ref{a}) and \ref{ce}) are very similar as in  \cite{TD} section 7,  so we will not repeat them. The proof of \ref{be}) is obvious since for every such $x$ it is not so hard to find $m$ such that $f^m(x) \in I$
%we can also prove that the only pair which can have some kind of chaos have to have 1 fix point, the proof would be again very similar to the one from  \cite{TD}.
\end{proof}

\begin{lem}\label{DC3lem}
The system $(\mathcal D,f)$ has a DC$3$ pair.
\end{lem}

\begin{proof}
%Thru the proof 
We will prove that the pair $(x,y)$, $y=\left(\frac13, \frac13\right)$, $ x= ( 1,0) $ is a DC$3$ pair.
For showing that it is a DC$3$ pair we need to find an interval such that  $ \Phi_{(f,x, y)}(\delta)<\Phi^*_{(f,x, y)}(\delta), $ for every $\delta$ in that interval. Since distribution functions are nondecreasing, it is enough to find 2 values  $a, b$ such that $a < b$ and  $\Phi_{(f,x, y)}(b)<\Phi^*_{(f,x, y)}(a)$
%$\le\Phi^*_{(f,x, y)}(b)$
.\\
Let's consider $b=1/2$. \\
Since the metric is maximal, the height will be irrelevant, it is enough to count the total number of spikes and the number of spikes on the right side of $\mathcal D$ at some concrete times - for half way through the levels. For $\Phi^*_{(f,x, y)}$ we will need an odd level + half of the next even level (because in even levels points move to the left), whereas for $\Phi_{(f,x, y)}$ the counting will go from the ``left'' and we will ``stop'' just before entering the $b-$neighborhood 
of $x$.\\
%horna
$$\Phi^*_{(f,x, y)}\left(0.5\right) = \displaystyle \lim _{k \to \infty} 
\left(
 \dfrac{ \frac12 l_{2k+2}}{\frac12 (l_{2k+2}+l_{2k+1})} \right)=
  \lim _{k \to \infty}
\left(\dfrac{ \frac12 (3^{(2k+2)} -1)}{\frac12 (3^{2k+2} -1 +3^{2k+1}-1)}\right) 
% \lim _{(2k+1) \to \infty}
%\left(\dfrac{ \frac12 (3^{(2k+2)} -1)}{\frac12 (3^{2k+2} -1 +3^{2k+1}-1)}\right)  
=\dfrac34$$ \\
%dolna
$$\Phi_{(f,x, y)}\left(0.5\right) = \displaystyle \lim _{k \to \infty} 
\left(\dfrac{  \frac12 l_{2k} }{\frac12 (l_{2k+1}+l_{2k})}\right)=\dfrac14.$$\\
%%%%%%%%
%%  ==== A ====
%%%%%%%
Let  $a=1/4$.
$$\Phi^*_{(f,x, y)}\left(0.25\right) = \displaystyle \lim _{k \to \infty} 
\left(
 \dfrac{ \frac14 l_{2k+2}}{\frac14 l_{2k+2}+ \frac34 l_{2k+1}} \right)=
  \lim _{k \to \infty}
\left(\dfrac{ (3^{(2k+2)} -1)}{ (3^{2k+2} -1) +3(3^{2k+1}-1)}\right) 
=\dfrac12$$
$\left[\Phi_{(f,x, y)}(0.5)<\Phi^*_{(f,x, y)}(0.25) \right]  \Rightarrow  \Phi_{(f,x, y)}(\delta)<\Phi^*_{(f,x, y)}(\delta) $ for every $\delta \in \langle 0.25 ; 0.5 \rangle $ and so the pair $(x,y)$ is DC3.
%
%   a=1/3
%
%$$\Phi^*_{(f,x, y)}\left(1/3\right) = \displaystyle \lim _{(2k+1) \to \infty} 
%\left(
% \dfrac{ \frac13 (l_{2k+2}+1) -1}{\frac13 (l_{2k+2}+1) -1+ \frac23 ( l_{2k+1} +1)} \right)= $$
%$$=\lim _{k \to \infty}
%\left(\dfrac{ \frac13 (3^{(2k+2)}) -1}{\frac13 (3^{2k+2}) +\frac23 (3^{2k+1})  -1}\right) 
%=\dfrac12$$ \\
%
\end{proof}

\begin{lem}\label{DC2lem}
The system $(\mathcal D,f)$  has no DC$2$-pairs.
\end{lem}

\begin{proof}
By Lemma \ref{spoj+} it is enough to check  pairs with 1 fixed point and 1 point of type $\left(z_j^{(n)} , \frac{1}{3^n}\right)$ (endpoints), moreover each such  point has a preimage in % $z_1^{(1)}
$\left(\frac13 ,\frac{1}{3}\right)$.
So we need to prove that for every $x \in I $ and $y=\left(\frac13 ,\frac{1}{3}\right)$   there is a $ \delta \in (0, 1)$ such that $\Phi^*_{(f,x, y)}(\delta) \neq 1.$ It is also clear that for $\delta \ge 1/2$ there are points $x$ for which   $\Phi_{(f,x, y)}(\delta) =\Phi^*_{(f,x, y)}(\delta) = 1,$ so we will search $\delta \in (0,\frac12).$ \\
\begin{figure}[H]
 \centering
\includegraphics[width=0.85\textwidth]{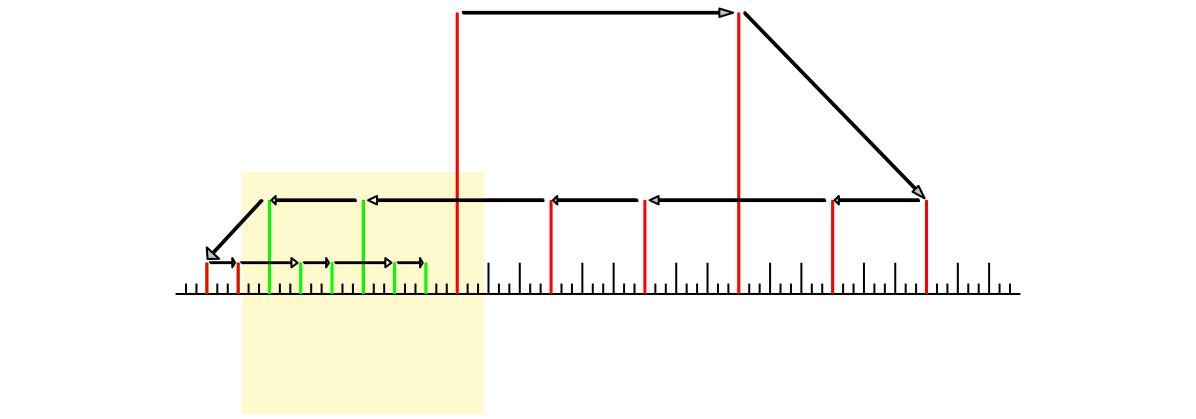}
\caption{ Illustration of a $\delta$-neighborhood for $\Phi^*_{(f,x, y)}(\delta)$.}
\label{dendriteDC3c}
\end{figure}
For $\delta \in \left(0,\frac12\right)$, $ y=\left(\frac13 ,\frac{1}{3}\right)$ and $x= \left(x_1 ,0\right)$, where $x_1 \in \langle 0, 1 \rangle$, there is an exact formula for $\Phi^*_{(f,x, y)}(\delta):$
%For such a $\delta$ we can write exact formula for $\Phi^*_{(f,x, y)}(\delta):$\\
%
%exact formula for upper distribution function
	\begin{equation}\label{fi*}    
		\Phi^*_{(f,x, y)}(\delta) =
%		\displaystyle\lim_{k \to \infty} \dfrac{a(x,\delta) l_{2k}}{b(x,\delta) l_{2k} + [1-b(x,\delta)] l_{2k-1}} = 
		\displaystyle\lim_{k \to \infty} \ \dfrac{a(x,\delta) \cdot l_{2k+1} + c(x,\delta,k) }{b(x,\delta) \cdot l_{2k+1} + \left[1-b(x,\delta)\right]\cdot l_{2k}}
			\end{equation}
			where
			$$
			a(x,\delta) = | \min\{(x_1 + \delta) , 1 \} - \max\{(x_1 - \delta) , 0 \} |,  $$
			
			$$
			-3^m a(x,\delta) -2\le c(x,\delta,k) \le 2 \mbox{\  where \ } m=\dfrac{  \log {\delta^{-1}} }{\log 3}%\mbox{ (c is an error term)}
			$$
			and
			$$
					b(x,\delta) =
					 \left \{   \begin{array}{l  l}
							x_1 + \delta , & \mbox{  if  $\ x_1 \in \langle 0, \frac12 \rangle$, } \\
							1-x_1 + \delta ,  & \mbox{  if   $\ x_1 \in \langle \frac12 , 1 \rangle $.}  \\ 
							\end{array} \right. \\	
	$$
Moreover, if we fix $\delta$, the function $u(x_1)=\Phi^*_{\left(f,(x_1,0), \left(\frac13, \frac13\right)\right)}(\delta)$ has 2 local maxima at $x_{1_1} = \delta,$ $ x_{1_2} = 1- \delta.$
Now, let $\delta=\frac14.$ Then $\Phi^*_{(f,x, y)}(0.25) \le \frac34 < 1 $ for every $(x,y)$ defined as above and so there are no DC2-pairs in $\mathcal D$ . 
\end{proof}

\subsection{DC3 and DC2$\frac12$}

Since the new type of DC (DC2$\frac12$) defined in \cite{RRH} was supposed to fix some  problems of DC3, (e.g. DC2$\frac12 \Rightarrow $ LY) and the above example has also LY-pairs, there was a hope that DC3 might imply at least DC2$\frac12$. Unfortunately, as the next lemma shows, even this weaker hope is not fulfilled on dendrites.
%
%\begin{definition*}\label{DC2.5}
%We can define both distribution functions at $0$ as limits:
% $\Phi_{(f, x, y)}(0)=  \displaystyle\lim_{\delta\to0^+}\Phi_{(f,x, y)}(\delta)$ and 
% $\Phi^*_{(f, x, y)}(0)=\displaystyle\lim_{\delta\to 0^+}\Phi^*_{(f,x, y)}(\delta)$. Then $(x, y)\in X^2$ is called \emph{distributionally scrambled of type $2\frac12$} (or a DC$2\frac12$ pair) if
%$$\Phi_{(f,x, y)}(0)<\Phi^*_{(f,x, y)}(0).$$
%We define also DC$2\frac12$ sets and DC$2\frac12$ systems in the same way as for the other versions of distributional chaos.
%\end{definition*}
%

\begin{lem}\label{DC2.5lem}
The system $(\mathcal D,f)$  has no DC$2\frac12$-pairs.
\end{lem}

\begin{proof}
The system  $(\mathcal D,f)$  has no DC$2\frac12$-pairs if $\Phi^*_{(f,x, y)}(0)=\Phi_{(f,x, y)}(0)$ for every pair $x,y \in \mathcal D.$ As in Lemma \ref{DC2lem}, it is enough to check pairs such that $x\in I$ and $y=\left(\frac13, \frac13\right)$. 
Since both distribution functions are nonnegative and nondecreasing, it is sufficient to show that
 for such a pair $(x,y)$, $\displaystyle\lim_{\delta \to 0^+} \Phi^*_{(f,x, y)}(\delta) = 0$.
% Then by the previous  discusion and lemmas  
%  $\displaystyle\Phi^*_{(f,x, y)}(0) =\Phi_{(f,x, y)}(0) $ for every $ x,y \in \mathcal D$ and so there is no DC2$\frac12$ pair.% \\
  
 Let  $x\in I$ and $y=\left(\frac13, \frac13\right)$. $\Phi^*_{(f,x, y)}$ is defined by (\ref{fi*}) in Lemma \ref{DC2lem}, then 
	\begin{equation*} 
		 \displaystyle\lim_{\delta \to 0^+} \Phi^*_{(f,x, y)}(\delta) =  
		  \displaystyle\lim_{\delta \to 0^+} \dfrac{3 a(x,\delta)}{1+2 b(x,\delta) } =
		  \end{equation*}
		  \begin{equation*}
		= \left \{   \begin{array}{l  l}							 \displaystyle\lim_{\delta \to 0^+} \dfrac {3(x_1 + \delta)}{1+2 (x+\delta)} = \displaystyle\lim_{\delta \to 0^+} \dfrac {3\delta}{1+2 \delta}   , & \mbox{  if  $\ x_1 = 0 $ } \\[1.5em]
		 
							\displaystyle\lim_{\delta \to 0^+} \dfrac {6\delta}{1+2 b(x,\delta)} ,  & \mbox{  if   $\ x_1 \in ( 0 , 1 ) $}  \\ [1.5em]		
												
							\displaystyle\lim_{\delta \to 0^+} \dfrac {3(1-x_1 + \delta)}{3-2x+2\delta}=\displaystyle\lim_{\delta \to 0^+} \dfrac {3\delta}{1+2\delta}  ,  & \mbox{  if   $\ x_1 =1 $}  \\ 							
							\end{array} \right\} =0. 
		%\displaystyle\lim_{k \to \infty} \ \dfrac{a(x,\delta) \cdot l_{2k+1} + c(x,\delta) }{b(x,\delta) \cdot l_{2k+1} + \left[1-b(x,\delta)\right]\cdot l_{2k}}
	\end{equation*}
	
 Then by the previous  discussion and lemmas,  
  $\displaystyle\Phi^*_{(f,x, y)}(0) =\Phi_{(f,x, y)}(0) $ for every $ x,y \in \mathcal D$ and so there is no DC2$\frac12$ pair.% \\
\end{proof}

%%%%
% Theorem DC3
%%%%%%
By combining the results of Lemmas \ref{DC3lem}, \ref{DC2lem} and \ref{DC2.5lem} we obtain the following theorem:

\begin{thm}\label{ThDC3}
DC3 implies neither DC2 nor  DC2$\frac12$ on dendrites.
\end{thm}

%\begin{proof}
%The  Theorem follows directly from Lemmas \ref{DC3lem}, \ref{DC2lem} and \ref{DC2.5lem}.  
%\end{proof}

\subsection{DC-pairs and uncountable DC-sets}\label{subsec:DCpairs}
It is easy to see that the  example in the above subsection also shows that existence of a DC3-pair implies  neither existence of an uncountable DC3-set nor an infinite DC3-set (as a simple corollary of Lemma \ref{spoj+} we get that the system does not have any DC3-triples). \\
However the whole construction above was  inspired by the example constructed in  \cite{TD} proving the next theorem and as a simple corollary of that theorem,  and the fact that every DC1-pair is also a LY-pair, we obtain a stronger result:
\begin{thm}[%Drwiega 
\cite{TD} Th 7.6] \label{TD76}
There exists a continuous self-map $f$ of a dendrite such that:
$f$ has a DC1-pair but does not have any infinite LY-scrambled set.
\end{thm}

\begin{cor}\label{DCpary}
Existence of a DC$i$ pair does not imply existence of an infinite DC$i$-set for any known $i-$type of distributional chaos $\left ( i \in \{1, 1.5, 2, 2.5, 3\}\right)$.
\end{cor}
%The question, if 
 %This example became an inspiration for the following dendrite $\mathcal D$ and dendrite map $f$:\\

%{\color{red}

\begin{note}
Notice, that the above construction of the dendrite $\mathcal D$ can be changed to have DC2 or DC1-pairs (one example for the DC1 case can be seen in \cite{TD}). All that we need to change is how many spikes we will add in each level between each 2 ``old'' spikes.
If we use any bounded number of ``new spikes'' between every 2 ``old spikes'', we will get DC3, if the number of added spikes  between every 2 ``old spikes'' will grow without bound we will get at least DC2.
But for this construction, there is a clear jump between DC3 and DC2, so that leaves us with the question whether it can be that DC2$\frac12$ implies DC2. 
\end{note}

%}

%%uncoutable DC3
%{\color{red}

\subsection{Uncountable DC3-set}\label{subsec:DC3un}
In subsection \ref{subsec:DC1} we just simply used known results on the 2-shift. One would think we can do the same for the case DC3 and DC2. Unfortunately even though many tried to find such a 2-shift,  to our knowledge there is no such  result. We tried too, but even though we were able to construct an infinite DC3-set% (see \ref{dodatok})
, it was not bigger than countable.
Fortunately, just when we were ready to give up, we found out that there is an example  of a subshift of the full 5-shift in \cite{WLF}, which has the required proprerties. And so we just need to construct a Gehman-like dendrite for the full 5-shift and its subshifts. We will denote this dendrite $\mathcal G_{5}$ and the corresponding map $g_{5}$.\\
We will not prove Lemma \ref{g5}, since the proof would be just a small modification of previous proofs for the classic Gehman dendrite. But let us  briefly recall how the map $g_5$ works:\\
We denote the branching points as follows: $c$ is the top point, then from left to right in the next level:
$c_{0}, c_{1}, c_{2}, c_{3}, c_{4}$, in the next level down : $c_{00}, c_{01}, c_{02}, c_{03}, c_{04}, c_{10},\dots , c_{44},$ and so on.
Similarly for the branches we start with $B_{0}=[c,c_0],\dots,B_{4}=[c,c_4]$,  and for every $n \in \mathbb N$ the branch $B_{i_1...i_{n+1}}=[c_{i_1...i_n},c_{i_1...i_{n+1}}]$, where each $i_j \in \{0, 1, 2, 3, 4\}$.
The map $g_5$ is defined such that:
$c$ is a fixed point, 
$B_i \rightarrow c$ and every  $B_{i_1...i_{n}} \rightarrow B_{i_2...i_{n}}$, in such a way that $c_{i_1...i_{n}} \rightarrow c_{i_2...i_{n}}$, and $c_i \rightarrow c$, where all $i$'s  are from the set $ \{0, 1, 2, 3, 4\}.$  The map on the limit set is the full 5-shift.

\begin{figure}[h]
 \centering
 \includegraphics[width=14.2cm]{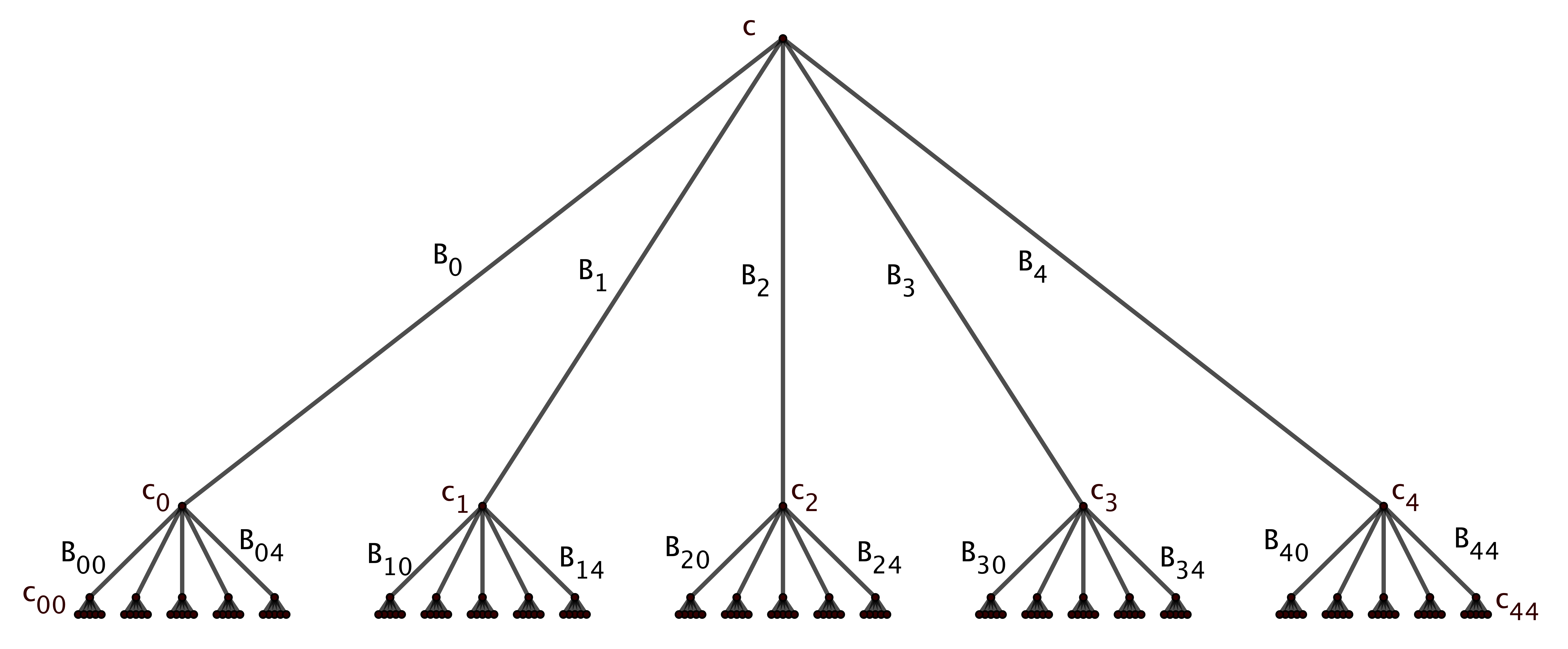}  
  \caption{Gehman style dendrite $\mathcal G_{5}$}
   \end{figure}

\begin{lem}\label{g5}
If $Y \subset  \{0, 1, 2, 3, 4\}^{\mathbb N_0}$
is a subshift then 
there is a subdendrite $\mathcal G_{5_Y}$ of the  dendrite $\mathcal G_5$ invariant under $g_5$.
Let $ \mathcal E_{5_Y}$ be the set of end points of $\mathcal G_{5_Y}$, then $( \mathcal E_{5_Y}, g |_{\mathcal E_{5_Y}} )$ is topologically conjugate to $(Y, \sigma)$ and all DC-pairs of  $\mathcal G_{5_Y}$ are contained in  $\mathcal E_{5_Y}$.
%the set $\mathcal G_X$ is an $g|_{\mathcal G_X}$-invariant subdendrite of the Gehman dendrite $\mathcal G$.
\end{lem}

\begin{thm}[\cite{WLF} sec. 5]\label{WLF}
There exists a system  $(Y,\sigma)$, such that $Y \subset   \{0, 1, 2, 3, 4\}^{\mathbb N_0}$ and has an uncountable DC3-set, but has no DC2-pair.
\end{thm}

As a combination of the previous theorems and lemma we obtain the following result:

\begin{thm}\label{ThDC2}
DC3 does not imply DC2 on dendrites in the sense of uncountable sets or pairs.
\end{thm}

%}

%%%%%%%%%
%       Podkova, nespocetna DC, omega
%%%%%%%%

%\section{:) uncountable DC and horseshoe}
\section{The strongest type of DC and other types of chaos}\label{sec:nesp.mn}
%====LEMMA===
%
In \cite{Koc} and other literature are shown different relationships betwen DC1$_2$ and other types of chaos. How will the situation change, if we want that some kind of chaos will imply DC1$_u$ (the strongest possible DC) instead of just pairs?
The most important example is the horseshoe, because in \cite{Koc} it is the only studied property 
%other carasteristic 
which implies DC1, but  it was shown just for DC1$_2$ .

\begin{lem}\label{shift}
The dynamical system $\left(\{0,1\}^{\mathbb N_0} , \sigma \right)$, where $\sigma$ is the shift map has an uncountable DC$1$ set.
\end{lem}
%
%===Dukaz===
%
\begin{proof} 
%metrika
\textbf{Definition of the metric:} To discuss DC chaos we have to specify a metric. Let $d$ be the metric on $\{0,1\}^{\mathbb N_0}$ given by $d(x,y) = {2^{-i}}$, where $i = \inf \{j: x_j \neq y_j\}$ and $x \in \{0, 1\}^{\mathbb N_0}$ represents an infinite sequence $x=x_0 x_1x_2 x_3\cdots$ of $0$s and $1$s. \\
%constraction of DC set
\textbf{Construction of the uncountable scrambled set ($\Lambda$):} \\
We can either use a tail equivalence relation which involves the axiom of choice  or  we may follow a more constructive approach and take $\Lambda = \nu \left( \lambda\left(\{0,1\}^{\mathbb N_0} \right)\right)$, where  $\lambda\!:~\{0,1\}^{\mathbb N_0} \to \{0,1\}^{\mathbb N_0}$ provides tail inequivalence $\lambda(x_0 x_1x_2 \cdots) =  x_0 \ x_0 x_1 \ x_0x_1x_2  \  \cdots$ so if $x \neq y$ then $\lambda(x)_i \neq \lambda(y)_i $ infinitely many times $\left[ \lambda(x) \not \thicksim \lambda(y)\right]$
and then $\nu \!:~\{0,1\}^{\mathbb N_0} \to \{0,1\}^{\mathbb N_0}$,  provides blocks (of $x_i$'s and $0$'s) of sufficient length : $\nu \left(x_0 x_1x_2 x_3x_4x_5\cdots \right)= 0^{a_1} \, x^{a_2} _0 \, 0^{a_3}  \, x^{a_4} _1 x^{a_5} _2 \, 0^{a_6}  \, x^{a_7} _3 x^{a_8} _4 x^{a_9}_5 \, 0^{a_{10}} \,  \cdots$, (where $0^{a_i}$ denotes $0 \, \ a_i$ times). Then by composition we will get: 
\begin{equation}\label{nu}
\nu \left( \lambda(x_0 x_1x_2 x_3\cdots)\right)= 0^{a_1} \, x^{a_2} _0 \, 0^{a_3}  \, x^{a_4} _0 x^{a_5} _1 \, 0^{a_6}  \, x^{a_7} _0 x^{a_8} _1 x^{a_9}_2 \, 0^{a_{10}} \,  \cdots,
\end{equation}
 where $\{ a_i\}_{i \in \mathbb N}$ is an increasing sequence  (eg: $i!^{i!}$) such that 
\begin{equation}\label{a_i} 
\mbox{ $\dfrac{b_n +n}{a_{n+1}} \to 0\ $ for $\, n \to \infty$, \,where $b_n = \sum_{i=1}^n a_i$.}
\end{equation}
 Since both $\lambda$ and $\nu$ are injective maps, %on $\{0,1\}^\mathbb{N}$,
 $\Lambda = \nu \left( \lambda\left(\{0,1\}^{\mathbb N_0}\right)\right)$ is an uncountable set, and we claim that $\Lambda$ is also a DC$1$ set for the shift map $\sigma $.\\
 Moreover the construction of $\Lambda$ motivates us to think about elements $\widehat{x} \in \Lambda$ in blocks, where the $i$-th block 
 $\left( B_i  = %\left(\widehat{x}_{ \sum_{l=1}^{i-1}a_l + 1} \ \dots \ \widehat{x}_{ \sum_{l=1}^{i}a_l }\right ) \right)$ 
 \left(\widehat{x}_{(b_{(i-1)}+1)} \ \dots \ \widehat{x}_{b_{i}} \right) \right)$
 has length $a_i$
 and for any element $\widehat{x} \in \Lambda$ there exists exactly one  $x \in \{0, 1\}^{\mathbb N_0}$ such that $\widehat{x} = \nu \circ \lambda (x)$ %, where $x \in \{0, 1\}^{\mathbb N_0}.$ (Notice that %since $\nu \circ \lambda$ is injective, 
($\widehat{x} \neq \widehat{y}$ implies $x\neq y.$)
\\So for any two $\widehat{x} \neq \widehat{y} \in \Lambda$ there are not just infinitely many $i \in \mathbb N_0$ such that $\widehat{x}_i \neq \widehat{y}_i$ but infinitely many blocks $B_i$
%(witch length coresponds with walues of subse $a_i$ 
such that $\widehat{x}_i \neq \widehat{y}_i$ for every $i$ in those blocks. Then for the calculation of $\Phi_{(\sigma,\widehat{x}, \widehat{y})}(\delta)$, the worst scenario is when we use $x,y$ which are different just in one coordinate or are very close (if there are more $j$'s such that $x_j \neq y_j$ or the first such $j$ is smaller, then  $\Phi_{(\sigma,\widehat{x}, \widehat{y})}(\delta)$ will be smaller/converge faster). For  $\Phi_{(\sigma,\widehat{x}, \widehat{y})}^*(\delta)$, the worst scenario is when we use $x,y$  which are different in every coordinate.\\
%Moreover the construction of $\Lambda$ motivates us to think about elements $\widehat{x} \in \Lambda$ in blocks of length $a_i$.\\
%
% VYPOCET DISTRIB.FUNKCII
%
%dolna distribucna funkcia
Calculation of \textbf{$\Phi_{(\sigma,\widehat{x}, \widehat{y})}(\delta)$} for $\delta \in (0,1\rangle$ (Instead of the definition we will use Note \ref{DCnote}.): \\
For any two $x \neq y \in \{0,1\}^{\mathbb N_0}$ there exists at least one $j$ such that $x_j \neq y_j $ (it does not matter which such $j$ will we use in the sequel, but for simplicity, let's take the smallest and denote it $j_0$). 
%Define $i(j) := (j+1) + \sum_{i=1}^{j+i}i$.
Then for $\widehat{x},\widehat{y}$ there exists a subsequence  $\{B_{i(j_0)}\}_{i \in \mathbb N} \subset  \{B_{i}\}_{i \in \mathbb N} $ of  blocks with sequence of lengths $\{a_{i(j)}\}_{i \in \mathbb N} $, such that
% $\widehat{x}, \widehat{y}$ will be different, where $\alpha(j)_n = (j+1) + \sum_{i=1}^{j+n}i $, 
%$\alpha(j)_1= (j+1) [\mbox{zeros blocks}] + \sum_{i=1}^{j+1}i [\mbox{blocks of } x_0, x_0x_1... x_0...x_j] $ \\ 
% $\alpha(j)_i= k\alpha_{i-1} + 1[\mbox{zeros block}] + (j+1)+(i-1) [\mbox{we $\forall$ need  j+1 "blocks of symbols" to get from } \\ x_0 \mbox{ to } x_j \mbox{+ i-1 "new", which we are adding before next 0 came }]  = \alpha_{i-1} + j+ i $ \\
% \Rightarrow $\alpha(j)_n= (j+1) + \sum_{i=1}^{j+n}i $
%\left(\sum_{l=1}^{\alpha(j)_n-1}a_l \ ,\sum_{l=1}^{\alpha(j)_n}a_l  \right \rangle $ for every $n \in \mathbb N $ 
 $\widehat{x}_j \neq \widehat{y}_j$ for every  $\widehat{x}_j ,\widehat{y}_j \in B_{i(j_0)}$ and so $d (\sigma^j(\widehat{x}), \sigma^j(\widehat{y})) = 1$ for those $j$'s. \\
 %(Construction of all subsequences is connected to the sequence $\{a_{i}\}$ define for the construction of the map $\nu$ in (\ref{nu}).)
%
By (\ref{a_i}) ${a_{i}}/{b_i} \to 1$ as $i \to \infty$
%and $B_i =  \left(b_{i-1}, b_i \right \rangle$\\
%Denote $b_n :=  \sum_{l=1}^{n}(a_l)$, then we know that ${a_{n}}/{b_n} \to 1$ as $n \to \infty$ 
and so \\
  $\displaystyle\lim_{i(j_0)\to\infty}\frac{1}{b_{i(j_0)}}\#\{b_{{i(j_0)}-1} \le k \le b_{i(j_0)};d(\sigma^k(\widehat{x}),\sigma^{k}(\widehat{y}))\ge \delta\}= \displaystyle\lim_{i(j_0) \to\infty}\frac{a_{i(j_0)}}{b_{i(j_0)}}=1$.\\
This shows that
 $$\Phi_{(\sigma,\widehat{x}, \widehat{y})}(\delta) \le 1-  \displaystyle\lim_{i(j_0)\to\infty}\frac{1}{b_{i(j_0)}}\#\{b_{{i(j_0)}-1} \le k \le b_{i(j_0)};d(\sigma^k(\widehat{x}),\sigma^{k}(\widehat{y}))\ge \delta\} =0 .$$

% horna distribucna funkcia
Calculation of \textbf{$\Phi^*_{(\sigma,\widehat{x}, \widehat{y})}(\delta)$} for $\delta \in (0,1\rangle$ (We get to use the blocks of 0's.) :\\
For every $\widehat{x} \in \Lambda$ there is a subsequence $\{B_{i_l}\}_{l \in \mathbb N}  \subset  \{B_{i}\}_{i \in \mathbb N} $ of blocks %of length $\{a_{\beta_n}\}_{n \in \mathbb N}$, where $\widehat{x}_i$ is $0$. I.e. 
such that $\widehat{x}_j=0$ for every $\widehat{x}_j \in B_{i_l}$
%\left(\sum_{l=1}^{\beta_n-1}a_l \ ,\  \sum_{l=1}^{\beta_n}a_l \right \rangle$, 
where $i_l = \sum_{j=1}^{l}j $. 
%$\beta_1 = 1$ and $\beta_n = \beta_{n-1}+n $
\\
Since $d(\widehat{x},\widehat{y}) \in \{2^{-m}\}_{m\in \mathbb N}$, the function $\Phi^*_{(\sigma,\widehat{x}, \widehat{y})}$ is constant on intervals $(2^{-m},2^{-m+1}\rangle$, where $m \in \mathbb N$, so it is enough to check   $\Phi^*_{(\sigma,\widehat{x}, \widehat{y})}(\delta)$ for $\delta \in 2^{-m}$. But for every $m \in \mathbb N$ there exists an $N \in \mathbb N$ such that for every $l \ge N, \ a_{i_l} > m$ and since 
$\displaystyle\lim_{i \to \infty} \frac{a_{i+1}}{b_i +i} = 1, $
 $$\Phi^*_{(\sigma,\widehat{x}, \widehat{y})}(\delta) \ge \displaystyle\lim_{l \to \infty}\frac{1}{b_{i_l}}
 \#\{b_{{(i_l)}-1} < k \le b_{i_l}-m \ ;d(\sigma^k(\widehat{x}),\sigma^{k}(\widehat{y})) < \delta\} =1 $$
 That shows $\Phi_{(\sigma,\widehat{x}, \widehat{y})}(\delta) =0 $ and 
 $ \Phi^*_{(\sigma,\widehat{x}, \widehat{y})}(\delta)=1$ for any $\widehat{x}, \widehat{y}\in \Lambda $ and $\delta \in (0,1\rangle$.

 %so and that $\nu \left(\lambda(x)\right) \not\thicksim \nu \left(\lambda(y)\right)$ for $x \neq y$. 
\end{proof}
 %$\hfill \Box \Box$ 
 
 %Poznamka
As a corollary of Lemma \ref{shift} we can strengthen the statement of Lemma 2 from \cite{Koc}.
%Lemma 2 - zosilnenie Lem2 z Koc. clanku
\begin{thm}\label{Lem2}
Let $f$ be a continuous self-map of a dendrite. If an iterate of $f$ has an arc horseshoe then $f$ is DC$1_u$ and $\omega$-chaotic.
\end{thm}
%
%===Dukaz L2===
%
\begin{proof} 
In \cite{Koc}  this lemma was proved for DC$1_2$. The author used the fact that if an iterate of the map $f$ has an arc horseshoe, then there is a set $D \subset X$ such that some iterate of $f$ restricted to $D$ is 2-to-1 semiconjugate to the shift map $\sigma$ on the space $\{0,1\}^{\mathbb N_0}$ which by \cite{HaM} and \cite{Li93} is DC$1_2$ and $\omega$-chaotic and thus $f$ is as well DC$1_2$ and $\omega$-chaotic. Moreover by Lemma \ref{shift}, $\left(\{0,1\}^{\mathbb N_0} , \sigma \right)$ is also DC$1_u$.
\end{proof}

\section{Open questions and remarks}\label{sec:QR}
While we showed that already on dendrites DC3 $\not \Rightarrow$ DC2 $\not \Rightarrow$ DC1 for any size of DC-scrambled set, nor do DC-pairs imply infinite DC-scrambled sets, there are at least 2 more important questions for DC on dendrites:
\begin{Q}
Do DC2$\frac12$ and DC2  coincide on dendrites?
 \end{Q}
  Or we can single out an even more obvious question (since DC3 $\not \Rightarrow$ DC2$\frac12$). 
  \begin{Q}
  Does DC3 imply at least LY on dendrites? 
  \end{Q}
To be clear, we are asking about the implications DC$3_2 \overset{?}{\Rightarrow}$ LY$_2$ and DC$3_u \overset{?}{\Rightarrow}$ LY$_u$,  since we know from \cite{TD} that DC$1_2 \not \Rightarrow$ LY$_\infty$.\\
Moreover we can ask if  any of the theorems in this paper would change, if we required the functions on dendrites to be monotone (ours are not).

\section*{Acknowledgments} \noindent 
This research was supported by grant SGS 16/2016 from  the Silesian University in Opava and RVO funding for IC47813059.\\
The author also wishes  to thank  Zden\v ek Ko\v can and Samuel Roth for their personal guidance, corrections and support. 

%\newpage
%%%=== Literatura ===

\end{document}